\documentclass[twoside, 12pt]{article}
\usepackage{mathrsfs}
\usepackage{amssymb, amsmath, mathrsfs, amsthm}
\usepackage{graphicx}
\usepackage{subcaption}
\usepackage{color}
\usepackage{tikz}
\usepackage[top=2cm, bottom=2cm, left=1.8cm, right=1.8cm]{geometry}
\usepackage{float, caption, subcaption}
\usepackage{diagbox}
\usepackage{algorithm}
\usepackage{algpseudocode}
\usepackage{amsmath}
\usepackage[colorlinks,
  linkcolor=blue,
 anchorcolor=blue,
 citecolor=blue]{hyperref} 
\input{epsf}

\newcommand{\E}{\mathbb{E}}
\newcommand{\Prob}{\mathbb{P}}

\newcommand{\cut}[1]{\|#1\|_\square}

\newcommand{\bd}{\begin{description}}
\newcommand{\ed}{\end{description}}
\newcommand{\bi}{\begin{itemize}}
\newcommand{\ei}{\end{itemize}}
\newcommand{\be}{\begin{enumerate}}
\newcommand{\ee}{\end{enumerate}}
\newcommand{\beq}{\begin{equation}}
\newcommand{\eeq}{\end{equation}}
\newcommand{\beqs}{\begin{eqnarray*}}
\newcommand{\eeqs}{\end{eqnarray*}}

\definecolor{DarkGreen}{rgb}{0.2, 0.6, 0.3}

\usepackage{multicol}
\usepackage{amsthm}
\usepackage{multirow}
\usepackage{makecell}
\usepackage{mathtools}

\newtheorem{theorem}{Theorem}[section]
\newtheorem{conjecture}{Conjecture}

\newtheorem{lemma}{Lemma}[section]

\newtheorem{claim}{Claim}

\begin{document}
\title{\textbf{A Solution to the Kohayakawa–Pr\"{o}mel–R\"{o}dl Conjecture} \footnote{Supported by the National Science Foundation of China
(Nos. 12471329 and 12061059), JSPS KAKENHI No. JP25KJ2077.}
}

\author{
Chuang Zhong\footnote{School of Mathematics and Statistis, Qinghai Normal University, Xining, Qinghai 810008, China. {\tt zhongchuang1998@163.com}}, \ \
Masaki Kashima \footnote{Corresponding author: Faculty of Science and Technology, Keio University, Yokohama, Kanagawa, Japan. {\tt  masaki.kashima10@gmail.com}}, \ \
Yaping Mao\footnote{Academy of Plateau
Science and Sustainability, and School of Mathematics and Statistis, Qinghai Normal University, Xining, Qinghai 810008,
China. {\tt yapingmao@outlook.com; myp@qhnu.edu.cn}}\ \ 
}
\date{}
\maketitle

\begin{abstract}
The induced Ramsey number $r_{\mathrm{ind}}(G,H)$ is defined as the minimum order of a graph $F$ such that any 2-coloring of its edges with red and blue leads to either a red induced copy of $G$ or a blue induced copy of $H$. 
In 1998, Kohayakawa, Pr\"omel, and R\"odl conjectured that for any fixed graph $G$, there exists a constant $f=f(G)$ such that the induced Ramsey number $r_{\mathrm{ind}}(G, H)$ is bounded by $|V(H)|^f$ for every graph $H$. 
In this paper, we completely solve the Kohayakawa--Pr\"omel--R\"odl conjecture, which states that for any graphs $G$ and $H$ with $|V(G)|=k\ge2$ and $|V(H)|=t\ge2$, $r_{\mathrm{ind}}(G,H)\le t^{60k^2+3}$.
This verifies the conjecture with $f(G)=60|V(G)|^2+3$.
\\[2mm]

\end{abstract}

\section{Introduction}
All graphs in this paper are undirected, finite and simple.
For a graph $G$ and a vertex $v \in V(G)$, its neighborhood $N_{G}(v)$ is the set of vertices adjacent to $v$, and its closed neighborhood is $N_{G}[v] = N_{G}(v) \cup \{v\}$. The degree of a vertex $v \in V(G)$ is $d_{G}(v) = |N_{G}(v)|$. 
For a set of vertices $S\subseteq V(G)$, $G[S]$ is the subgraph of $G$ induced by $S$.
For two disjoint subsets $S$ and $T$ of $V(G)$, let $E_G(S,T)$ denote the set of edges of $G$ with one end in $S$ and the other end in $T$, and let $e_G(S,T)=|E_G(S,T)|$.
Moreover, we write $E_G(v,T)$ or $e_G(v,T)$ instead of $E_G(\{v\},T)$ or $e_G(\{v\},T)$ for a vertex $v$.
Undefined concepts or notation can be found in \cite{BM}.

A cornerstone of classical Ramsey theory is the study of Ramsey numbers. For given graphs $G$ and $H$, the \textit{Ramsey number} $r(G,H)$ is defined as the minimum integer $n$ such that every red/blue edge-coloring of the complete graph $K_{n}$ contains either a red copy of $G$ or a blue copy of $H$ as a subgraph.
A natural and highly restrictive variant of this concept is the induced Ramsey problem. For graphs $F$, $G$, and $H$, we write $F\xrightarrow{\mathrm{ind}}(G,H)$ if every red/blue edge-coloring of $F$ yields either a red induced copy of $G$ or a blue induced copy of $H$. Here, a red (resp. blue) induced copy of $G$ (resp. $H$) refers to an induced subgraph of $F$ that is isomorphic to $G$ (resp. $H$) and has all its edges colored red (resp. blue).

The foundational fact that such a graph $F$ always exists for any choice of $G$ and $H$ was independently proved in the 1970s by Deuber \cite{Deuber75}, Erd\H{o}s, Hajnal, and Pósa \cite{EHP75}, and R\"{o}dl \cite{R73,R76}. This seminal result sparked extensive research into \textit{induced Ramsey numbers}. For any graphs $G$ and $H$, the induced Ramsey number, denoted by $r_{\mathrm{ind}}(G,H)$, is defined as the minimum number of vertices in a graph $F$ satisfying $F\xrightarrow{\mathrm{ind}}(G,H)$. 
Moreover, when $G = H$, it is simply referred to as the induced Ramsey number of $G$, and is denoted by $r_{\mathrm{ind}}(G)$.
Since in the case of complete graphs the induced subgraph is the same as the subgraph, it is obvious that $r_{\mathrm{ind}}(K_{m},K_{n}) = r(K_{m},K_{n})$. 

Much of the early focus was on the diagonal case $r_{\mathrm{ind}}(G)$. Erd\H{o}s \cite{E75} conjectured an exponential upper bound $r_{\mathrm{ind}}(G) \leq 2^{cn}$. A major breakthrough towards this conjecture was made by Kohayakawa, Pr\"{o}mel, and R\"{o}dl \cite{KPR98}, who established a quasi-exponential upper bound of $2^{cn(\log n)^2}$. This was later improved by Conlon, Fox, and Sudakov \cite{CFS12} to $2^{cn \log n}$, bringing the bound even closer to Erd\H{o}s's original conjecture. Fox and Sudakov \cite{FS08} further developed a unified approach that yielded linear bounds for specific structures such as paths and stars. In 2025, Arag\~{a}o, Campos, Dahia, Filipe, and Marciano \cite{ACDFM25} 
fully resolved this conjecture.

However, the off-diagonal case $r_{\mathrm{ind}}(G, H)$ presents different challenges. In 1998, Kohayakawa, Pr\"omel, and R\"odl proposed the following conjecture regarding the polynomial boundedness of induced Ramsey numbers when one graph is fixed.

\begin{conjecture}{\upshape \cite{KPR98}}\label{con-1}
For any graph $G$, there is a constant $f=f(G)$ that depends only on $G$ such that, for any graph $H$ on $n$ vertices, we have
$$
r_{\mathrm{ind}}(G, H) \leq n^f.
$$    
\end{conjecture}

Towards this conjecture, Kohayakawa, Pr\"omel, and R\"odl \cite{KPR98} proved that $r_{\mathrm{ind}}(G,H)\le t^{Ck\log q}$ when
$|V(G)|=k\le |V(H)|=t$ and $q=\chi(H)\ge2$,
where $C>0$ is an absolute constant. 
To further investigate Conjecture \ref{con-1}, they verified the conjecture for every \(G\in\mathcal S\), where \(\mathcal S\) is the smallest class of finite graphs containing the $1$-vertex complete graph and closed under taking disjoint unions and joins. Since trees are not covered by $\mathcal{S}$, they also addressed them as a fundamental case, proving that for any $k$-vertex tree $T$ and $t$-vertex graph $H$, $r_{\mathrm{ind}}(T, H) \leq c k^2 t^4 \left( \frac{\log(k t^2)}{\log \log \log(k t^2)} \right)^2$. 
Cycles have also been studied extensively.
Haxell, Kohayakawa, and {\L}uczak~\cite{HKL95} proved that the
induced size-Ramsey number of $C_\ell$ is linear in $\ell$,
which implies that
$r_{\mathrm{ind}}(C_\ell,C_\ell)=O(\ell)$.
Schaefer and Shah~\cite{SS03} gave explicit constructions yielding
$r_{\mathrm{ind}}(C_\ell,B)
\le |V(B)|^{\lceil\log_2\ell\rceil+4}$
and
$r_{\mathrm{ind}}(C_\ell,T)
\le |V(T)|^{\lceil\log_2\ell\rceil+1}$,
where $\ell\ge3$, $B$ is bipartite, and $T$ is a tree.
For a varying fan graph $F_r$, Zhong et al.~\cite{ZKMZ25}
proved that $r_{\mathrm{ind}}(G,F_r)\le C_G r^2$
for every fixed graph $G$, where $C_G$ depends only on $G$.
For other results on induced Ramsey numbers, 
see~\cite{AG20star,AG20,Gorgol01,Gorgol19,GL02,KS06}.

In this paper, we resolve Conjecture~\ref{con-1} by proving
the following general bound.

\begin{theorem}\label{thm:main}
Let $G$ and $H$ be graphs with $|V(G)|=k\ge2$ and $|V(H)|=t\ge2$.
Then
 $r_{\mathrm{ind}}(G,H)\le 7t^{60k^2}\le t^{60k^2+3}$.
\end{theorem}
Therefore, Theorem~\ref{thm:main}
verifies Conjecture~\ref{con-1} with
 $f(G)=60|V(G)|^2+3$.

\section{Deterministic counting lemmas}\label{sec:counting}

For a positive integer $r$, write $[r]=\{1,\ldots,r\}$.
For a graph $R$ and $U\subseteq V(R)$, write $e_R(U)=e(R[U])$.
All logarithms are natural unless another base is
specified. An empty product is interpreted as~$1$.

For a finite set $V$ of size $N$ and a real matrix $A:V\times V\to\mathbb R$,
define the normalized cut norm by
 $$\cut{A}
 =\frac1{N^2}\max_{S,T\subseteq V}
   \left|\sum_{x\in S,\,y\in T}A(x,y)\right|.$$
The sets $S$ and $T$ need not be disjoint. A scalar subtracted from a matrix denotes the constant matrix on all of $V\times V$, including the diagonal.

We begin with a weighted form of the cut-norm estimate,
which allows us to control sums involving functions
taking values in $[0,1]$.

\begin{lemma}\label{lem:weighted-cut}
Let $A:V\times V\to\mathbb R$ be a real matrix and $|V|=N$. For all functions
$f,g:V\to[0,1]$, we have
 $$\left|\sum_{x,y\in V}A(x,y)f(x)g(y)\right|
 \le \cut{A}N^2.$$
\end{lemma}

\begin{proof}
For a set $S\subseteq V$, let $\mathbf 1_S$ denote its indicator
function, that is,
\[
\mathbf 1_S(x)=
\begin{cases}
1,& x\in S,\\
0,& x\notin S.
\end{cases}
\]
We first express the functions $f$ and $g$ in terms of indicator
functions of their level sets.

Fix $x\in V$. Since $0\le f(x)\le 1$, for $a\in [0,1]$, we have
\[
\mathbf 1_{\{f(x)\ge a\}}
=
\begin{cases}
1,&a\le f(x),\\
0,&f(x)<a.
\end{cases}
\]
Therefore,
$\int_0^1 \mathbf 1_{\{f(x)\ge a\}}\,da
=
\int_0^{f(x)}1\,da
=
f(x)$, and thus
$f(x)
=
\int_0^1 \mathbf 1_{\{f(x)\ge a\}}\,da$.
Similarly, for every $y\in V$, we have
$g(y)
=
\int_0^1 \mathbf 1_{\{g(y)\ge b\}}\,db$.
Thus,
\begin{align*}
f(x)g(y)
=
\left(
\int_0^1
\mathbf 1_{\{f(x)\ge a\}}\,da
\right)
\left(
\int_0^1
\mathbf 1_{\{g(y)\ge b\}}\,db
\right)=
\int_0^1\int_0^1
\mathbf 1_{\{f(x)\ge a\}}
\mathbf 1_{\{g(y)\ge b\}}
\,da\,db.
\end{align*}
Substituting this identity into the weighted sum gives
\begin{align*}
\sum_{x,y\in V}A(x,y)f(x)g(y)
&=
\sum_{x,y\in V}
A(x,y)
\int_0^1\int_0^1
\mathbf 1_{\{f(x)\ge a\}}
\mathbf 1_{\{g(y)\ge b\}}
\,da\,db\\
&=
\int_0^1\int_0^1
\sum_{x,y\in V}
A(x,y)
\mathbf 1_{\{f(x)\ge a\}}
\mathbf 1_{\{g(y)\ge b\}}
\,da\,db.
\end{align*}
Here we may interchange the finite sum and the integrals since $V$ is finite.

For fixed $a,b\in[0,1]$, define the level sets
$S_a=\{x\in V:f(x)\ge a\}$ and
$T_b=\{y\in V:g(y)\ge b\}$.
Then
$\mathbf 1_{\{f(x)\ge a\}}
=
\mathbf 1_{S_a}(x)$ and
$\mathbf 1_{\{g(y)\ge b\}}
=
\mathbf 1_{T_b}(y)$.
Hence
$\mathbf 1_{S_a}(x)\mathbf 1_{T_b}(y)=1$
if and only if both $x\in S_a$ and $y\in T_b$, implying that
\[
\sum_{x,y\in V}
A(x,y)
\mathbf 1_{\{f(x)\ge a\}}
\mathbf 1_{\{g(y)\ge b\}}
=
\sum_{\substack{x\in S_a\\ y\in T_b}}
A(x,y).
\]
Combine this with the above equation, it follows that
\[
\sum_{x,y\in V}A(x,y)f(x)g(y)
=
\int_0^1\int_0^1
\sum_{\substack{x\in S_a\\ y\in T_b}}
A(x,y)
\,da\,db.
\]

By the definition of the cut norm, every pair of subsets
$S,T\subseteq V$ satisfies
\[
\left|
\sum_{\substack{x\in S\\y\in T}}
A(x,y)
\right|
\le
\|A\|_{\square}N^2.
\]
Applying this with $S=S_a$ and $T=T_b$, and then using the triangle inequality, we obtain
\begin{align*}
\left|
\sum_{x,y\in V}A(x,y)f(x)g(y)
\right|
&\le
\int_0^1\int_0^1
\left|
\sum_{\substack{x\in S_a\\y\in T_b}}
A(x,y)
\right|
\,da\,db\\
&\le
\int_0^1\int_0^1
\|A\|_{\square}N^2
\,da\,db\\
&=
\|A\|_{\square}N^2.
\end{align*}
This proves the lemma.
\end{proof}

Using the preceding estimate, we now compare induced
embedding counts with ordinary homomorphism counts by
replacing the nonedge factors one at a time.

Let $\Gamma$ be a graph on $V$, and define its \emph{nonedge matrix} $Z_\Gamma$ by
 $$Z_\Gamma(x,y)=
 \begin{cases}
  1,&x\ne y\text{ and }xy\notin E(\Gamma),\\
  0,&\text{otherwise}.
 \end{cases}$$
In particular, $Z_\Gamma(x,x)=0$ for any $x\in V$.
Let $R$ be a subgraph of $\Gamma$, and denote its
adjacency matrix by
\[
R(x,y)=
\begin{cases}
1,&xy\in E(R),\\
0,&\text{otherwise}.
\end{cases}
\]
In particular, $R(x,x)=0$ for any $x\in V$.
For a graph $G$ with $V(G)=[k]$, let $\overline{E}(G)=\binom{[k]}2\setminus E(G)$ and $m=|\overline{E}(G)|$.
Define
\[
\operatorname{hom}(G,R)=
 \sum_{(x_1,\ldots,x_k)\in V^k}
 \prod_{ij\in E(G)}R(x_i,x_j)
\]
and
\[
I_G(R,Z_\Gamma)=
\sum_{(x_1,\ldots,x_k)\in V^k}
\left(
\prod_{ij\in E(G)}R(x_i,x_j)
\prod_{ij\in\overline{E}(G)}Z_\Gamma(x_i,x_j)\right).
\]
The first expression counts graph homomorphisms, with no injectivity requirement. In the second expression, if $x_i=x_j$ for two distinct indices, then $Z_\Gamma(x_i,x_j)=0$ and the monomial vanishes.
Thus $I_G(R,Z_\Gamma)$ counts the labeled induced embeddings of $G$ into $\Gamma$ whose edges all belong to $R$.

The next lemma adapts the standard cut-norm telescoping argument to a finite matrix setting. For the usual counting lemma, see~\cite[Section~4.5]{Zhao23}. We include the whole proof just to show clearly that only the nonedge factors are changed.

\begin{lemma}\label{lem:mixed-count}
Let $\Gamma$ be an $N$-vertex graph, let $R$ be a spanning subgraph of $\Gamma$, and suppose that $0<\sigma\le1$ and
$\cut{Z_\Gamma-\sigma}\le\varepsilon$. 
Then every graph $G$ with vertex set $[k]$ satisfies
\[
\left|I_G(R,Z_\Gamma)-\sigma^m\operatorname{hom}(G,R)\right|
 \le m\varepsilon N^k,
\]
where $m=\binom{k}{2}-e(G)$ is the number of nonedges of $G$.
\end{lemma}

\begin{proof}
If $m=0$, then $G=K_k$. In this case there are no nonedge factors in the definition of $I_G(R,Z_\Gamma)$, and hence
$I_G(R,Z_\Gamma)=\operatorname{hom}(G,R)
=
\sigma^0\operatorname{hom}(G,R)$.
Thus the assertion is immediate. 
We may therefore assume that $m\ge1$.

Denote the nonedges of $G$ as $\{e_1,e_2,\ldots,e_m\}$ and let
$e_h=a_hb_h$.
For each $s\in\{0,1,\ldots,m\}$, define
\[
T_s=
\sum_{(x_1,\ldots,x_k)\in V^k}
\left(
\prod_{ij\in E(G)}R(x_i,x_j)
\sigma^s
\prod_{h=s+1}^{m}
Z_\Gamma(x_{a_h},x_{b_h})\right).
\]
The quantity $T_s$ should be viewed as an intermediate count in which the first $s$ nonedge factors have already been replaced by the constant $\sigma$, while the remaining $m-s$ nonedge factors are still kept as $Z_\Gamma$.
In particular, when $s=0$, no nonedge factor has been replaced, and so
\[
T_0
=
\sum_{(x_1,\ldots,x_k)\in V^k}
\left(
\prod_{\{u,v\}\in E(G)}R(x_u,x_v)
\prod_{h=1}^{m}
Z_\Gamma(x_{a_h},x_{b_h})\right)
=
I_G(R,Z_\Gamma).
\]
On the other hand, when $s=m$, every nonedge factor has been replaced by $\sigma$, and hence
\[
T_m
=
\sigma^m
\sum_{(x_1,\ldots,x_k)\in V^k}
\prod_{\{u,v\}\in E(G)}R(x_u,x_v)
=
\sigma^m\operatorname{hom}(G,R).
\]
It is therefore enough to estimate the change between $T_{s-1}$ and $T_{s}$ at each single replacement.

Fix $s\in[m]$ and set $e_s=ab$.
By the definition of $T_{s-1}$ and $T_s$, the two expressions differ only in the factor corresponding to this nonedge; that is, the factor is
$Z_\Gamma(x_i,x_j)$ in $T_{s-1}$, whereas it has been replaced by $\sigma$ in $T_s$. 
Therefore
\begin{align}
T_{s-1}-T_s
=
\sum_{(x_1,\ldots,x_k)\in V^k}
\left(
\prod_{ij\in E(G)}R(x_i,x_j)
\times
\sigma^{s-1}
\bigl(
Z_\Gamma(x_a,x_b)-\sigma
\bigr)
\times
\prod_{h=s+1}^{m}
Z_\Gamma(x_{a_h},x_{b_h})\right).
\label{eq:replacement-difference}
\end{align}

We now fix all coordinates
$x_r$ with $r\in[k]\setminus\{a,b\}$
and regard only $x_a$ and $x_b$ as variables.

\begin{claim}\label{clm:factor-separation}
Apart from the factor
$Z_\Gamma(x_a,x_b)-\sigma$,
no other factor in \eqref{eq:replacement-difference} depends
simultaneously on both $x_a$ and $x_b$.
\end{claim}

\begin{proof}
Every factor appearing in \eqref{eq:replacement-difference}, except for the constant factor $\sigma^{s-1}$, is associated with an unordered pair of distinct vertices of $G$.

Now $e_s=ab$ is a nonedge of $G$, so the edge product contains no factor of the form $R(x_a,x_b)$.
Among the nonedge factors, the unique factor corresponding to the pair $ab$ is the one being replaced at the $s$th step, namely $Z_\Gamma(x_i,x_j)-\sigma$.
Thus, every other factor can involve at most one of the two
variables $x_a$ and $x_b$. Thus, each remaining factor depends either only on $x_i$, only on $x_j$, or on neither of them. 
\end{proof}

By Claim~\ref{clm:factor-separation}, after the other $k-2$
coordinates have been fixed, we may collect all factors depending only on $x_i$ into a function $f:V\to[0,1]$,
all factors depending only on $x_j$ into a function $g:V\to[0,1]$, and all factors depending on neither $x_i$ nor $x_j$, together with the constant factor $\sigma^{s-1}$, into a constant $c\in[0,1]$.
Thus, for this fixed choice of the other coordinates, the contribution to $T_{s-1}-T_s$ has the form $c\sum_{x_a,x_b\in V}
\bigl(Z_\Gamma(x_a,x_b)-\sigma\bigr)f(x_a)g(x_b)$.
Since every value of $R$ and $Z_\Gamma$ is either $0$ or $1$, while $0<\sigma\le1$, then every product formed from such factors lies in $[0,1]$.

Since $\cut{Z_\Gamma-\sigma}\le\varepsilon$, we can apply Lemma~\ref{lem:weighted-cut} with the matrix
$A=Z_\Gamma-\sigma$ to obtain
\[
\left|
\sum_{x_i,x_j\in V}
\bigl(
Z_\Gamma(x_i,x_j)-\sigma
\bigr)
f(x_i)g(x_j)
\right|
\le
\varepsilon N^2.
\]
Because $0\le c\le1$, we have $$c\sum_{x_i,x_j\in V}
\bigl(Z_\Gamma(x_i,x_j)-\sigma\bigr)
f(x_i)g(x_j)\le\varepsilon N^2.$$
This estimate holds for every fixed choice of the remaining $k-2$ coordinates. 
There are exactly $N^{k-2}$ possible choices for these coordinates. Summing the above estimate over all of them yields
$|T_{s-1}-T_s|
\le
N^{k-2}\cdot\varepsilon N^2
=
\varepsilon N^k$.
Thus
$T_0-T_m
=
\sum_{s=1}^{m}(T_{s-1}-T_s)$,
and hence
$|T_0-T_m|
\le
\sum_{s=1}^{m}|T_{s-1}-T_s|
\le
m\varepsilon N^k$.
Recalling that
$T_0=I_G(R,Z_\Gamma)$
and
$T_m=\sigma^m\operatorname{hom}(G,R)$,
we obtain
$\left|
I_G(R,Z_\Gamma)
-
\sigma^m\operatorname{hom}(G,R)
\right|
\le
m\varepsilon N^k$
\end{proof}

The preceding lemma is entirely deterministic once the host graph
$\Gamma$ is fixed. In particular, it holds simultaneously for every
spanning subgraph $R\subseteq\Gamma$, and no probabilistic independence
between the red graph $R$ and the host graph $\Gamma$ is required. This is
important in our Ramsey application, since the red graph is determined
only after an arbitrary coloring of the already chosen host.

We also stress that the intermediate quantities $T_s$ are summed over
all of $V^k$, without imposing injectivity. For $T_0=I_G(R,Z_\Gamma)$,
injectivity is automatic because both $R$ and $Z_\Gamma$ have zero
diagonal; if $x_i=x_j$, then the factor corresponding to the pair
$\{i,j\}$ vanishes. After a nonedge factor $Z_\Gamma(x_i,x_j)$ is
replaced by the constant $\sigma$, this automatic exclusion of
collisions may disappear. Thus the intermediate $T_s$ should not be
interpreted as induced embedding counts. Allowing all tuples in $V^k$
is precisely what permits the remaining factors in a single replacement
step to separate as $c f(x_i)g(x_j)$.

To apply the preceding counting estimate, we also need
a lower bound on the number of red homomorphisms.
The following lemma provides such a bound under a
local edge-density condition.

\setcounter{claim}{0}

\begin{lemma}\label{lem:local-count}
Let $k\ge2$ and $0<\rho\le1/2$. Suppose that an $N$-vertex graph $R$
satisfies
$e_R(U)\ge\rho|U|^2$
for every $U\subseteq V(R)$
with $|U|\ge\rho^{k-1}N$.
Then every graph $G$ on $k$ vertices satisfies
\[
\operatorname{hom}(G,R)\ge\operatorname{hom}(K_k,R)
 \ge\rho^{k(k+1)/2}N^k
 \ge\rho^{k^2}N^k.
\]
\end{lemma}

\begin{proof}
We first establish the following claim.

\begin{claim}\label{clm:many-large-degree}
Let $W\subseteq V(R)$ satisfy
$|W|\ge\rho^{k-1}N$.
Then there are at least $\rho|W|$ vertices of $W$ whose degree in the induced subgraph $R[W]$ is at least $\rho|W|$ .
\end{claim}

\begin{proof}
By the assumption of the theorem, we have $e_R(W)\geq \rho|W|^2$.
Let $D=\{v\in W\mid d_{R[W]}(v)\ge\rho|W|\}$.
We prove that $|D|\ge\rho|W|$.
Since each vertex in $W\setminus D$ has degree less than $\rho|W|$ in $R[W]$ and every vertex in $W$ has degree less than $|W|$ in $R[W]$, by handshaking lemma, we have
$$2\rho|W|^2\leq 2e_R(W)=\sum_{v\in W}d_{R[W]}(v)<(|W|-|D|)\rho|W|+|D||W|\leq \rho|W|^2+|D||W|.$$
This implies $|D|>\rho|W|$, as desired.
\end{proof}

We now construct many ordered $k$-tuples spanning a copy of $K_k$ in $R$. 
Let $W_0=V(R)$. Since $|W_0|=N\ge \rho^{k-1}N$, Claim~\ref{clm:many-large-degree} can apply to $W_0$. 
Hence at least $\rho|W_0|=\rho N$ vertices of $W_0$ have degree at least $\rho|W_0|$ in $R[W_0]$. 
Choose one such vertex $x_1$ and set $W_1=W_0\cap N_R(x_1)$.
For each $j=2,3,\ldots ,k$, we recursively apply Claim~\ref{clm:many-large-degree} to $W_{j-1}$ to find a vertex $x_j\in W_{j-1}$ with $|W_{j-1}\cap N_R(x_j)|\geq \rho|W_{j-1}|$ and set $W_j=W_{j-1}\cap N_R(x_j)$.
Note that we have $|W_{j-1}|\geq \rho^{j-1}N\geq \rho^{k-1}N$ in each step, so Claim~\ref{clm:many-large-degree} can apply.

By the construction, it follows that $x_j\in \bigcap_{\ell=1}^{j-1}N_R(x_\ell)$ for all $j\in \{2,3,\ldots ,k\}$, and hence $\{x_1,x_2,\ldots,x_k\}$ is a clique of $R$.
Furthermore, since we have $\rho|W_{j-1}|\geq \rho^jN$ different choice of $x_j$ for each $j$th step, $\operatorname{hom}(K_k,R)$, the number of ordered copies of $K_k$ in $R$ is bounded by
\[
\operatorname{hom}(K_k,R)\geq \prod_{j=1}^{k}\rho^j N=\rho^{\frac{k(k+1)}{2}}N^k\geq \rho^{k^2}N^k,
\]
where the last inequality follows from $k\geq 2$ and $0<\rho\leq 1$. 

For any graph $G$ with $k$ vertices, each ordered copy of $K_k$ in $R$ contains an ordered copy of $G$ in itself, and thus we have $\operatorname{hom}(G,R)\geq \operatorname{hom}(K_k,R)$, as desired.
\end{proof}

\section{The projective-plane}\label{sec:host}

We now introduce the random host graph.
For a prime $p$, let $N=p^2+p+1$.
A finite projective plane of order $p$ is an incidence structure $\mathcal P=(V,\mathcal L)$ with its point set as $V$ and its line set as $\cal L$ with the following properties
(see~\cite{Hirschfeld1998}):
\begin{itemize}
\item $|V|=N$ and $|\mathcal L|=N$;
\item every line $L\in\mathcal L$ contains exactly $p+1$ points;
\item every point $v\in V$ lies on exactly $p+1$ lines;
\item any two distinct lines intersect in exactly one point
\item any two distinct points determine a unique line.
\end{itemize}
It is well known that finite projective planes exist whenever $p$ is a prime power.

We will use the following elementary incidence calculation.

\begin{lemma}\label{lem:incidence}
Let $\mathcal{P}=(V,\mathcal{L})$ be a projective plane of order $p$.
For $U\subseteq V$, if $|U|\ge64p$, then at least $\frac{3N}{4}$ lines satisfy
\begin{equation}\label{eq:large-line-intersection}
 |U\cap L|\ge\frac{|U|}{4p}.
\end{equation}
\end{lemma}

\begin{proof}
Counting incidences with points in $U$ gives $\sum_{L\in\cal L}|U\cap L|=(p+1)|U|$.
Counting ordered pairs of distinct points in $U$ gives $\sum_{L\in\cal L}|U\cap L|(|U\cap L|-1)=|U|(|U|-1)$, since every such pair belongs to a unique line. 
Consequently, $\sum_{L\in\cal L}|U\cap L|^2=\sum_{L\in \mathcal{L}}|U\cap L|(|U\cap L|-1)+\sum_{L\in \mathcal{L}}|U\cap L|=|U|^2+p|U|$.

Now we set $\mu=\frac{(p+1)|U|}{N}$.
Since $(p+1)^2=N+p$, we have
\begin{align*}
    \sum_{L\in \mathcal{L}}(|U\cap L|-\mu)^2
    &=\sum_{L\in \mathcal{L}}|U\cap L|^2-2\sum_{L\in \mathcal{L}}|U\cap L|\frac{(p+1)|U|}{N}+\sum_{L\in \mathcal{L}}\left(\frac{(p+1)|U|}{N}\right)^2\\
    &=|U|^2+p|U|-\frac{2(p+1)^2|U|^2}{N}+\frac{(p+1)^2|U|^2}{N}\\
    &=|U|^2+p|U|-\frac{(N+p)|U|^2}{N}=p|U|\left(1-\frac{|U|}{N}\right)\leq p|U|,
\end{align*}
where the last inequality follows from $|U|\leq |V|=N$.
Suppose that $|U|\geq 64p$.
Each line $L$ with $|U\cap L|<\frac{\mu}{2}$ contributes to the above sum by at least $\frac{\mu^2}{4}$, so the number of such lines is at most $\frac{4p|U|}{\mu^2}$.
By the definition of $\mu$, we have $\mu\geq \frac{|U|}{2p}$, and hence $$\frac{4p|U|}{\mu^2}\leq \frac{16p^3}{|U|}\leq \frac{16p^3}{64p}=\frac{p^2}{4}<\frac{N}{4}.$$
Thus, at least $\frac{3N}{4}$ lines satisfies $|U\cap L|\geq \frac{\mu}{2}\geq \frac{|U|}{4p}$, as required.
\end{proof}

The randomness will be placed on incidences rather than on points.
Fix a graph $H$ with $V(H)=[t]$. For each incidence $(x,L)$ with
$x\in L$, choose a label $\lambda_L(x)$ independently and uniformly
from $[t]$. There are $N(p+1)$ independent labels. In particular,
$\lambda_L(x)$ and $\lambda_{L'}(x)$ are independent when $L\ne L'$.

For distinct $x,y\in V$, let $L(x,y)$ be their unique common line.
Define a random graph $\Gamma$ on $V$ by
\begin{equation}\label{eq:host-rule}
 xy\in E(\Gamma)~\text{if and only if}~
 \lambda_{L(x,y)}(x)\lambda_{L(x,y)}(y)\in E(H).
\end{equation}
This is the projective-plane random graph introduced in the form used
here in~\cite[Section~2.1]{KPR98}. For $L\in\cal L$ and $i\in[t]$, set $L_i=\{x\in L\mid\lambda_L(x)=i\}$.
For distinct $i,j\in[t]$, the bipartite graph between $L_i$ and $L_j$ is complete if $ij\in E(H)$ and empty if $ij\notin E(H)$. Moreover, each $L_i$
is an independent set, since $H$ has no loops.
Therefore, every choice of one point from each of the $t$ nonempty classes $L_i$ gives an induced copy of $H$ in $\Gamma$. 

Set
\begin{equation}\label{eq:sigma}
 \sigma=1-\frac{2e(H)}{t^2}.
\end{equation}
For distinct $x,y$, $xy$ is an edge of $\Gamma$ if $x$ and $y$ map to some edge of $H$ by $\lambda_{L(x,y)}$.
Since there are $2e(H)$ possibilities of such maps, for the nonedge matrix, we conclude that
\begin{equation}\label{eq:nonedge-expectation}
 \E Z_\Gamma(x,y)=\sigma.
\end{equation}
Also, we have
\begin{equation}\label{eq:sigma-lower}
 \sigma\ge1-\frac{t(t-1)}{t^2}=\frac1t.
\end{equation}
The diagonal entries of $Z_\Gamma$ are zero rather than having $\sigma$; their contribution will be retained explicitly in the
cut-norm estimate.

\section{Properties of the random host}\label{sec:host-properties}

We first show that all sufficiently large vertex sets meet every label
class on many lines. This property depends only on the random labels,
not on any subsequent edge-coloring.

\begin{lemma}\label{lem:labels}
Let $t\ge2$, let $b\ge1$, and suppose that
$p\ge512t^{b+2}$.
With probability at least $1-e^{-2N}$, every $U\subseteq V$ with $|U|\ge \frac{N}{t^b}$ satisfies that at least $\frac{N}{2}$ lines $L\in\cal L$ satisfy
\begin{equation}\label{eq:balanced-line}
 |U\cap L_i|\ge\frac{|U|}{8tp}
 \quad\text{for every }i\in[t].
\end{equation}
\end{lemma}

\begin{proof}
Fix a set $U\subseteq V$ with $|U|\ge \frac{N}{t^b}$.
Since $N\ge p^2$ and $p\ge512t^{b+2}$, we have $|U| \ge \frac{p^2}{t^b} \ge 64p$.
Let $$\mathcal Q(U) = \left\{ L\in\mathcal L: |U\cap L|\ge\frac{|U|}{4p} \right\}.$$
By Lemma~\ref{lem:incidence}, $|\mathcal Q(U)|\ge\frac{3N}{4}$.
Note that $\mathcal Q(U)$ depends only on $U$ and the incidence
structure of the projective plane, before any random labels are exposed.

Now fix a line $L\in\mathcal Q(U)$ and a label $i\in[t]$. Since the
labels $\{\lambda_L(x)\mid x\in L\}$ are independent and uniformly distributed on $[t]$, it means that the random variable
$X_i=|U\cap L_i|$ has distribution $X_i\sim\operatorname{Bin}(|U\cap L|,1/t)$, that is, $\mu_i=\frac{|U\cap L|}{t}$.
Since $L\in\mathcal Q(U)$, $\mu_i\ge\frac{|U|}{4tp}$.
Thus $\frac{|U|}{8tp}\le\frac{\mu_i}{2}$, and hence $\left\{ X_i<\frac{|U|}{8tp} \right\} \subseteq \left\{ X_i<\frac{\mu_i}{2} \right\}$.
By the Chernoff bound, we have $\Pr(X_i<\frac{\mu_i}{2})\le e^{-\mu_i/8}$, and hence 
$$\Prob\left( |U\cap L_i|<\frac{|U|}{8tp} \right) \le \Prob\left(|U\cap L_i|<\frac{\mu_i}{2}\right) \le \exp \left(-\frac{\mu_i}{8}\right) \le \exp\left(-\frac{|U|}{32tp}\right).$$

If $L$ fails \eqref{eq:balanced-line}, then there must exist some
$i\in[t]$ such that $|U\cap L_i|<\frac{|U|}{8tp}$.
Thus, by the union bound over the $t$ labels, we have
$\Prob\left(
L\text{ fails }\eqref{eq:balanced-line}
\right)
\le
t\exp\left(-\frac{|U|}{32tp}\right)$.
No independence among the random variables
$|U\cap L_1|,\ldots,|U\cap L_t|$ is needed here.
Set $s=\frac{N}{t^b}$ and $\eta_0 = t\exp\left(-\frac{s}{32tp}\right)$.
Since every set under consideration satisfies $|U|\ge s$, then $\Prob\left( L\text{ fails }\eqref{eq:balanced-line} \right) \le \eta_0$ for every $L\in\mathcal Q(U)$.
Moreover, since $N\ge p^2$, we have $s=\frac{N}{t^b}\ge\frac{p^2}{t^b}$, and hence $\eta_0 \le t\exp\left(-\frac{p}{32t^{b+1}}\right)$.
Using $p\ge512t^{b+2}$, we obtain that $\eta_0\le te^{-16t}<1$.

For this fixed set $U$, the failure events corresponding to distinct
lines are independent. Indeed, whether a line $L$ satisfies
\eqref{eq:balanced-line} depends only on the incidence labels
$\{\lambda_L(x)\mid x\in L\}$.
If $L\ne L'$, these are disjoint collections of point--line incidence
variables. Even when $L$ and $L'$ meet at a point $x$, the labels
$\lambda_L(x)$ and $\lambda_{L'}(x)$ are independent by construction.

Suppose that fewer than $\frac{N}{2}$ lines satisfy
\eqref{eq:balanced-line}. Since $|\mathcal Q(U)|\ge\frac{3N}{4}$,
it follows that at least $\frac{3N}{4}-\frac N2=\frac N4$
lines in $\mathcal Q(U)$ must fail
\eqref{eq:balanced-line}. 
Set $r=\lceil \frac{N}{4}\rceil$.
Using the binomial expansion, we have
\begin{align*}
    \Prob\left(\text{at least }r\text{ lines in }\mathcal Q(U) \text{ fail}~\eqref{eq:balanced-line}\right)
    =\sum_{\ell=r}^{N}\binom{N}{\ell}\eta_0^k(1-\eta_0)^{N-k}\leq \sum_{\ell=r}^N\binom{N}{\ell}\eta_0^r\leq 2^N\eta_0^r\leq 2^N\eta_0^{N/4}.
\end{align*}
Note that there are at most $2^N$ subsets of $V$. Thus a union bound over all
sets with $|U|\ge s$ shows that the probability that there exists some
$U\subseteq V$ with $|U|\ge s$
for which fewer than $\frac{N}{2}$ lines satisfy
\eqref{eq:balanced-line} is at most $4^N\eta_0^{N/4}$.
Since $\eta_0 = t\exp\left(-\frac{s}{32tp}\right)$, $N\ge p^2$ and $p\ge512t^{b+2}$, we obtain that
$$4^N\eta_0^{N/4} \le \exp\left[ N\left( \log4+\frac14\log t-4t \right) \right] \le e^{-2N}$$ with $t\ge2$.

Thus, with probability at least $1-e^{-2N}$,
every set $U\subseteq V$ with $|U|\ge\frac{N}{t^b}$
has at least $\frac{N}{2}$ lines satisfying
\eqref{eq:balanced-line}.
\end{proof}

The conclusion of Lemma~\ref{lem:labels} is a property of a realized labeled host. Since the union bound ranges over every vertex set, it also applies to sets selected after observing the labels or after choosing an edge-coloring.

We now extract a deterministic consequence of this property.

\begin{lemma}\label{lem:red-density}
Suppose that a realization of the labeled host graph satisfies the conclusion of Lemma~\ref{lem:labels}. 
In every red--blue edge-coloring of this host
with no blue induced copy of $H$, the red graph $R$ satisfies
$e_R(U)\ge\frac{|U|^2}{128t^2}$
when $|U|\ge \frac{N}{t^b}$.
\end{lemma}

\begin{proof}
Fix such a coloring and a set $U$ with $|U|\ge N/t^b$.
For a line $L$ satisfying~\eqref{eq:balanced-line}, let $W_i=U\cap L_i$ and $q=\frac{|U|}{8tp}$.
Each $W_i$ has at least $q>0$ points and is therefore nonempty.
Choose one point independently and uniformly from each $W_i$.
The resulting transversal is an induced copy of $H$ in the host.
Since no such copy is entirely blue, every transversal contains at
least one red edge.

For $ij\in E(H)$, let $r_{ij}$ be the number of red edges between $W_i$ and $W_j$. 
Since each edge counted as $r_{ij}$ contributes to $\prod_{\ell\notin \{i,j\}}|W_\ell|$ many transversals, we have $$\sum_{ij\in E(H)}\left(r_{ij}\prod_{\ell\notin \{i,j\}}|W_\ell|\right)\geq \prod_{1\leq \ell\leq t}|W_\ell|,$$ and hence
\[
 1\le\sum_{ij\in E(H)}\frac{r_{ij}}{|W_i||W_j|}
 \le\frac1{q^2}\sum_{ij\in E(H)}r_{ij}
 =\frac{e_R(U\cap L)}{q^2}.
\]
Thus, each line satisfying~\eqref{eq:balanced-line} contributes at least $q^2$ red edges in $U$. 
Every unordered pair of points lies on exactly one line, so no red edge is counted twice when we sum over lines.
By Lemma~\ref{lem:labels}, at least $\frac{N}{2}$ lines satisfy~\eqref{eq:balanced-line}, and hence
\[
 e_R(U)\ge \frac{N}{2}q^2=
 \frac N2\left(\frac{|U|}{8tp}\right)^2
 =\frac{N}{p^2}\frac{|U|^2}{128t^2}
 \ge\frac{|U|^2}{128t^2}.
\]
If $H$ has no edges, the transversal is automatically blue, and hence
the hypothesis of a coloring without a blue induced $H$ is impossible.
This also covers that case.
\end{proof}

We next establish a second property of the random host graph, namely
that its nonedge matrix $Z_\Gamma$ is close, in cut norm, to the constant
matrix of density $\sigma$. This property is proved independently of the
label-distribution property established. We
derive a separate failure probability for it and later combine the two
estimates by a union bound.

\begin{lemma}[McDiarmid~\cite{McD89}]
\label{lem:bounded-differences}
Let $X_1,\ldots,X_M$ be independent random variables taking values in
finite sets, and let $F=f(X_1,\ldots,X_M)$ be real-valued. Suppose that
changing only the $i$th coordinate changes $f$ by at most $c_i$, where
$c_i\ge0$. If $\sum_i c_i^2>0$, then, for every $a>0$,
\[
 \Prob(|F-\E F|\ge a)
 \le2\exp\left(-\frac{2a^2}{\sum_{i=1}^M c_i^2}\right).
\]
If all $c_i$ are zero, then $F$ is constant on the support and the
same tail probability is zero.
\end{lemma}

\begin{lemma}\label{lem:nonedge-cut}
Let $\Gamma$ be the random graph defined in~\eqref{eq:host-rule}, and let $\sigma$ be as in~\eqref{eq:sigma}. 
If 
 $\varepsilon N\ge2$
and
 $\varepsilon^2p\ge64$,
then
\[
\Prob\bigl(\cut{Z_\Gamma-\sigma}>\varepsilon\bigr)
 \le2e^{-2N}.
\]
\end{lemma}

\begin{proof}
Fix two subsets $S,T\subseteq V$, and define
$X_{S,T}=\sum_{x\in S,\,y\in T}Z_\Gamma(x,y)$.
Thus $X_{S,T}$ counts the ordered pairs $(x,y)\in S\times T$ with
$x\ne y$ for which $xy$ is a nonedge of the host graph $\Gamma$.

For every pair of distinct vertices $x,y$, we have
$\E Z_\Gamma(x,y)=\sigma$
by~\eqref{eq:nonedge-expectation}. On the other hand,
$Z_\Gamma(x,x)=0$ for every $x\in V$. 
There are $|S||T|$ ordered pairs in $S\times T$. 
Exactly $S\cap T$ of them are diagonal pairs $(x,x)$, which satisfies $Z_\Gamma(x,x)=0$. 
Every other pair $(x,y)$ satisfies $\E Z_\Gamma(x,y)=\sigma$. Hence, by linearity of expectation,
\begin{equation}\label{eq:cut-mean}
\E X_{S,T}
=
\sigma\bigl(|S||T|-|S\cap T|\bigr).
\end{equation}

We next estimate how much $X_{S,T}$ can change when one incidence label is changed. 
Recall that the random host graph is determined by the
$N(p+1)$ independent labels $\lambda_L(x)$ with $x\in L$.
Fix one incidence $(x,L)$ and change only the value of $\lambda_L(x)$, leaving all other incidence labels unchanged.

The adjacency status of a pair containing $x$ can be affected only when the other endpoint also lies on $L$. 
Since $L$ contains $p+1$ points, there are exactly $p$ such unordered pairs $\{x,y\}$ with $y\in L\setminus \{x\}$.
No pair determined by a different line can be affected because the label of $x$ on another line $L'$ is the different random variable $\lambda_{L'}(x)$.

Each affected unordered pair $\{x,y\}$ contributes at most two ordered
pairs to the sum defining $X_{S,T}$, namely $(x,y)$ and $(y,x)$.
Therefore, changing one incidence label changes $X_{S,T}$ by at most
$2p$.

We may now apply Lemma \ref{lem:bounded-differences} to the
$N(p+1)$ independent incidence labels. Since the change caused by any
one label is at most $2p$, for every $a>0$ we obtain
\begin{equation}\label{eq:fixed-cut-tail}
\Prob\bigl(
 |X_{S,T}-\E X_{S,T}|\ge a
\bigr)
\le
2\exp\left(
-\frac{2a^2}{4p^2N(p+1)}
\right).
\end{equation}

The quantity appearing in the cut norm is not
$X_{S,T}-\E X_{S,T}$, but rather
$X_{S,T}-\sigma|S||T|$.
The two centres differ only because of the diagonal terms. Indeed, by~\eqref{eq:cut-mean},
$\bigl|\E X_{S,T}-\sigma|S||T|\bigr| = \sigma|S\cap T| \le N$.
Hence, by the triangle inequality, we have 
$$|X_{S,T}-\E X_{S,T}| \ge \bigl|X_{S,T}-\sigma|S||T|\bigr| - \bigl|\E X_{S,T}-\sigma|S||T|\bigr| \ge \bigl|X_{S,T}-\sigma|S||T|\bigr|-N.$$

Suppose now that
$\bigl|X_{S,T}-\sigma|S||T|\bigr| > \varepsilon N^2$.
Then
$|X_{S,T}-\E X_{S,T}| > \varepsilon N^2-N$.
Since $\varepsilon N\ge2$, we have
$N\le\frac{\varepsilon N^2}{2}$,
and therefore
$\varepsilon N^2-N \ge \frac{\varepsilon N^2}{2}$.
Consequently,
$\bigl|X_{S,T}-\sigma|S||T|\bigr| > \varepsilon N^2$
implies
$|X_{S,T}-\E X_{S,T}| > \frac{\varepsilon N^2}{2}$.

Applying~\eqref{eq:fixed-cut-tail} with
$a=\varepsilon N^2/2$, we obtain, for this fixed pair $S,T$,
\[
\Prob\left( \bigl|X_{S,T}-\sigma|S||T|\bigr| > \varepsilon N^2 \right) \le \Prob\left(|X_{S,T}-\E X_{S,T}|>\frac{\varepsilon N^2}{2}\right) \le 2\exp\left( -\frac{\varepsilon^2N^3}{8p^2(p+1)} \right).
\]

We now pass from a fixed pair $S,T$ to all pairs of subsets.
By the definition of the normalized cut norm, we have 
\[
\cut{Z_\Gamma-\sigma} = \frac1{N^2} \max_{S,T\subseteq V} \left| \sum_{x\in S,\,y\in T} \bigl(Z_\Gamma(x,y)-\sigma\bigr) \right|.
\]
Since
$\sum_{x\in S,\,y\in T} \bigl(Z_\Gamma(x,y)-\sigma\bigr) = X_{S,T}-\sigma|S||T|$,
the event
$\cut{Z_\Gamma-\sigma}>\varepsilon$
is equivalent to the existence of subsets $S,T\subseteq V$ such that
$\bigl|X_{S,T}-\sigma|S||T|\bigr| > \varepsilon N^2$.

There are $2^N$ choices for $S$ and independently $2^N$ choices for
$T$, so there are at most $4^N$ ordered pairs $(S,T)$. Hence, by the
union bound,
\begin{equation}\label{eq:cut-union-bound}
\Prob\bigl(
\cut{Z_\Gamma-\sigma}>\varepsilon
\bigr)
\le
2\cdot4^N
\exp\left(
-\frac{\varepsilon^2N^3}{8p^2(p+1)}
\right).
\end{equation}

It remains to simplify the exponent. Since
$N=p^2+p+1\ge p^2$
and $p+1\le2p$, we have
$$\frac{N^3}{8p^2(p+1)} = N\frac{N^2}{8p^2(p+1)} \ge N\frac{p^4}{8p^2\cdot2p} = \frac{Np}{16}.$$
Therefore the right-hand side of
\eqref{eq:cut-union-bound} is at most
$2\exp\left[ N\left( \log4-\frac{\varepsilon^2p}{16} \right) \right]$.
Finally, the assumption $\varepsilon^2p\ge64$ gives
$\frac{\varepsilon^2p}{16}\ge4$,
and hence, since $\log4<2$,
$\log4-\frac{\varepsilon^2p}{16} \le \log4-4 < -2$.
Thus
$\Prob\bigl( \cut{Z_\Gamma-\sigma}>\varepsilon \bigr) \le 2e^{-2N}$,
as required.
\end{proof}

The edges of the random host graph are not independent. The independent
variables in the proof of Lemma~\ref{lem:nonedge-cut} are exclusively
the point--line labels. The influence bound $2p$ accounts for all
ordered pairs that a single label can affect, including when $S$ and
$T$ overlap.

\section{Proof of the main theorem}\label{sec:main-proof}

\begin{proof}[Proof of Theorem~\ref{thm:main}]
Fix $k,t\ge2$, a graph $G$ with $|V(G)|=k$, and a graph $H$ with
$V(H)=[t]$. If $H$ has no edges, the edgeless graph on $t$ vertices
already contains a blue induced copy of $H$ in every edge-coloring,
so the theorem holds. We may therefore assume that $e(H)\ge1$.

Choose the parameters
$b=10k$, $A=30k^2$, $\rho=t^{-10}$ and $\varepsilon=t^{-12k^2}$.
By Bertrand's postulate, there is a prime $p$ with
$t^A\le p\le2t^A$.
Take the random host  graph of Section~\ref{sec:host} on
$N=p^2+p+1$ vertices.

Since $k\ge2$, then $A-b-2=30k^2-10k-2\ge9$.
This, together with $t\ge 2$, implies that $p\ge t^A\ge t^{b+2}t^9\ge512t^{b+2}$.
Thus Lemma~\ref{lem:labels} can apply. 
Since $\varepsilon^2p\ge t^{6k^2}\ge64$ and
$\varepsilon N\ge\varepsilon p^2\ge t^{48k^2}\ge2$, Lemma~\ref{lem:nonedge-cut} also applies. Recall that Lemma~\ref{lem:labels} fails with probability at most $e^{-2N}$, while Lemma~\ref{lem:nonedge-cut} fails with probability at most $2e^{-2N}$.
Therefore, by the union bound, the probability that both conclusions
hold simultaneously is at least
 $1-e^{-2N}-2e^{-2N}=1-3e^{-2N}>0$.
Hence there exists at least one realization of the incidence labels for which both conclusions hold. Fix such a realization and let $\Gamma$ denote the resulting host graph.
In particular, Lemma~\ref{lem:labels} holds for this fixed host graph and $\cut{Z_\Gamma-\sigma}\le\varepsilon$.

Consider any red--blue coloring of $E(\Gamma)$, and let $R$ be its red spanning subgraph. If the coloring contains a blue induced copy of $H$, then we are done. 
Otherwise, by Lemma~\ref{lem:red-density}, we have
$e_R(U)\ge\frac{|U|^2}{128t^2}$ for every $U\subseteq V$ with $|U|\ge N/t^{10k}$.
Since $t^8\ge128$ for $t\ge2$, we have $\frac{1}{128t^2}\ge\rho$.
Also, $\frac{N}{t^{10k}}=\rho^kN$. Therefore,
 $e_R(U)\ge\rho|U|^2$
when $|U|\ge\rho^kN$.
To apply Lemma~\ref{lem:local-count}, we need the same inequality for
all sets satisfying
$|U|\ge\rho^{k-1}N$.
Since $0<\rho<1$, we have
$\rho^{k-1}\ge\rho^k$.
Therefore every set with $|U|\ge\rho^{k-1}N$ also satisfies
$|U|\ge\rho^kN$. Then
 $\operatorname{hom}(G,R)\ge\rho^{k^2}N^k=t^{-10k^2}N^k$.

Let
$m=\binom{k}{2}-e(G)$,
so that $m$ is the number of nonedges of $G$. By
Lemma~\ref{lem:mixed-count}, we can obtain that
$\left|I_G(R,Z_\Gamma)-\sigma^m\operatorname{hom}(G,R)\right|\le m\varepsilon N^k$, that is,
$I_G(R,Z_\Gamma)\ge \sigma^m\operatorname{hom}(G,R)-m\varepsilon N^k$.
Since $\sigma\ge\frac{1}{t}$ and
$\varepsilon=t^{-12k^2}$, it follows that
\begin{align}
I_G(R,Z_\Gamma)
&\ge
\sigma^m\operatorname{hom}(G,R)
-
m\varepsilon N^k
\notag\\
&\ge
\left(
t^{-m}t^{-10k^2}
-
mt^{-12k^2}
\right)N^k
\notag\\
&=
t^{-12k^2}
\left(
t^{2k^2-m}-m
\right)N^k.\notag
\end{align}
Since
$m\le\binom{k}{2}<k^2$,
we have
$2k^2-m>k^2$.
Since $t\ge2$, then
$t^{2k^2-m}>t^{k^2}\geq 2^{k^2}$.
Moreover,
$m\le\binom{k}{2}<k^2<2^{k^2}$
for every $k\ge2$. Hence
$t^{2k^2-m}>m$.
Then
$I_G(R,Z_\Gamma)>0$.

We now explain explicitly why this implies the existence of a red
induced copy of $G$ in the host graph $\Gamma$. By definition, we have 
$$I_G(R,Z_\Gamma)
=
\sum_{(x_1,\ldots,x_k)\in V(\Gamma)^k}
\left(
\prod_{ij\in E(G)}R(x_i,x_j)
\prod_{ij\in\overline E(G)}
Z_\Gamma(x_i,x_j)
\right).$$
Every factor in every summand is either $0$ or $1$. 
Since the total sum is positive, there exists at least one tuple
$(x_1,\ldots,x_k)\in V(\Gamma)^k$ such that $$\prod_{ij\in E(G)}R(x_i,x_j)
\prod_{ij\in\overline E(G)}
Z_\Gamma(x_i,x_j)=1.$$
Then we have $R(x_i,x_j)=1$ for each $ij\in E(G)$, implying that $x_ix_j\in R$.
Also, for each $ij\in \overline E(G)$, we have $Z_\Gamma(x_i,x_j)=1$, implying that $ij\notin E(G)$.
In particular, since either $R(x,x)=0$ or $Z_\Gamma(x,x)=0$ for any vertex $x\in V$, $(x_1,\ldots ,x_k)$ are pairwise distinct $k$ vertices.
Therefore, $(x_1,\ldots ,x_k)$ induces a red copy of $G$, and $\Gamma\xrightarrow{\mathrm{ind}}(G,H)$.

Since
$p\le2t^A$, we have
$N=p^2+p+1\le 4t^{2A}+2t^A+1\le 7t^{2A}$.
Since $A=30k^2$, this gives
$N\le7t^{60k^2}$.
Also, $t\ge2$, we have $7\le t^3$, and hence
$7t^{60k^2}\le t^{60k^2+3}$.
Hence
$r_{\mathrm{ind}}(G,H)\le 7t^{60k^2}\le t^{60k^2+3}$.
All the estimates above hold for every $k,t\ge2$, so no additional asymptotic or small-parameter cases remain.
\end{proof}

\end{document}